\documentclass{amsart}
\usepackage{amsmath,amssymb,amsthm}

\newtheorem{theorem}{Theorem}
\newtheorem{proposition}[theorem]{Proposition}
\newtheorem{lemma}[theorem]{Lemma}

\theoremstyle{definition}

\theoremstyle{remark}

\newcommand{\QQ}{\mathbb{Q}}

\newcommand{\ZZ}{\mathbb{Z}}

\newcommand{\CC}{\mathbb{C}}

\newcommand{\Kcyc}{K^{\mathrm{cyc}}}

\newcommand{\ord}{\operatorname{ord}}

\title[PCF polynomials over cyclotomic closures]{Postcritically finite polynomials over cyclotomic closures of number fields}
\author{Patrick Ingram and Jorge Mello}
\date{\today}

\begin{document}
\maketitle

\begin{abstract}
We prove that, for any number field $K$ and and $d\geq 2$, there are at most finitely many monic, centred, post-critically finite polynomials of degree $d$ with coefficients in $\Kcyc$, the extension of $K$ by all roots of unity.
\end{abstract}

\section{Introduction}\label{sec:intro}

A polynomial $f(z)\in \CC[z]$ is \emph{post-critically finite (PCF)} if and only if the orbit under $f$ of each critical point of $f$ is finite. It follows from the compactness of the connectedness locus in the moduli space of polynomials that every PCF polynomial is conjugate, by a linear change of variables, to a monic polynomial with coefficients in $\overline{\QQ}$. It turns out that the algebraic degree of the coefficients offers a stratification of conjugacy classes of PCF polynomials of a given degree. Specifically, the main result of~\cite{Ingram} shows that the coefficients of PCF polynomials written in monic centred form lie in a set of bounded height, and so in particular one has that there are only finitely many PCF polynomials defined over and given finite extension $K/\QQ$, up to change of variables.

An analogous situation occurs with torsion points on an abelian variety. If $A$ is an abelian variety defined over a number field $K$, then $A(K)$ contains only finitely many torsion points,  by a standard argument involving heights. Ribet~\cite{ribet} showed something much stronger, namely that $A(\Kcyc)$ contains only finitely many torsion points, where $\Kcyc$ is the cyclotomic closure of $K$, the infinite algebraic extension obtained by adjoining all roots of unity to $K$. Our main result shows something similar for PCF polynomials.

\begin{theorem}\label{thm:main}
Let $K$ be a number field and $d\geq 2$. There are only finitely many monic centred
 PCF  polynomials of degree $d$ with coefficients
in $\Kcyc$.
\end{theorem}

We recall that a polynomial is \emph{centred} if and only if its roots sum to zero. Since every polynomial $f(z)\in K[z]$ is conjugate over $K$ to a centred polynomial, and two monic centred polynomials are conjugate only by $z\mapsto \zeta z$ with $\zeta^{d-1}=1$, Theorem~\ref{thm:main} is equivalent to the statement that there are only finitely many conjugacy classes of monic PCF polynomials with coefficients in $\Kcyc$.

Note it follows from Theorem~\ref{thm:main} that if $B$ is an affine curve, and $f(z)\in K[B, z]$ is a non-isotrivial family of monic polynomials whose coefficients are regular functions on $B$, then only finitely many specializations $f_t$ for $t\in B(\Kcyc)$ are PCF. For a family with two independent critical orbits, there are only finitely many PCF specializations even over $\CC$ by results of Favre and Gauthier~\cite{FG}, but for families with only one symmetry class of infinite critical orbit (such as $z^d+c$) there will be infinitely many PCF specializations over $\overline{\QQ}$.

Our result is not the first that proves finiteness for PCF parameters over infinite algebraic extensions. Noytaptim and Petsche~\cite{np} showed that if $z^d+c$ is PCF with $c\in\QQ^\mathrm{tr}$, the maximal totally real extension of $\QQ$, then $c=0$; $c=-1$ and $d$ is even; or $c=-2$ and $d=2$. We also note that there are some previous results in arithmetic dynamics over cyclotomic and abelian closures, which overlap with our techniques~\cite{oy, chen, foz}.


The proof of Theorem~\ref{thm:main} is by contradiction and does not seem to produce any effective bounds. It follows from work of the first author~\cite{Ingram} that the coefficients of a PCF map in monic, centred form are adelically bounded. In other words, there is for each place $v\in M_\QQ$ a $C_v\geq 1$ depending only on $d\geq 2$ such that if  \[f(z)=z^d+a_{d-2}z^{d-2}+\cdots + a_0\]
with $a_i\in \overline{\QQ}$ is PCF, then $|a_i|_v\leq C_v$ for every extension of $|\cdot|_v$ to $\overline{\QQ}$, and $C_v=1$ for all but finitely many places of $\QQ$. Indeed, this applies not just to the coefficients of $a_i$, but to the elements of the critical orbits of $f$ as well. If $a_i\in \Kcyc$, then for any $n, m\geq 0$
\begin{equation}\label{eq:Wintro}W_{n, m}:=\prod_{f'(c)=0}(f^n(c)-f^m(c))\in\Kcyc,\end{equation}
where the product is taken over all critical points with multiplicity, is similarly bounded by an adelic constant that depends only on $d$.

Assuming for exposition that the $a_i$ and $W_{n, m}$ under consideration are algebraic integers, a result of Dvornicich and Zannier~\cite{DZ} extending earlier work of Loxton~\cite{Loxton} tells us that each coefficient $a_i\in\Kcyc$ and each $W_{n, m}$ associated to such an example can be written as a linear combination of a bounded number of roots of unity, with coefficients from a finite subset of $K$. Writing down the equations that express the $a_i$ and $W_{n, m}$ as linear combinations, and equating each of finitely many of the $W_{n, m}$ to the appropriate polynomials in the $a_i$, we obtain for each PCF polynomial a point on some variety whose coordinates are all roots of unity. Laurent's work~\cite{Laurent} on Lang's conjecture then gives us a torsion coset of a positive-dimensional torus $T$ contained in this variety. Restricted to this torsion coset and pulled back to $T$, the $a_i$ and $W_{n, m}$ are represented by Laurent polynomials. Given the setup so far, these Laurent polynomials are lacunary in the sense that each is the sum of a number of monomials which is bounded just in terms of $d$ (but independent of $n$ and $m$). A multivariable variant of a lemma of Haj\'os now tells us that the order of vanishing of $W_{n, m}$ on this torus is bounded by some constant that depends only on $d$, but crucially not on $n$ or $m$, unless $W_{n, m}=0$ identically.

This runs afoul of some results in complex dynamics. Any algebraic family of this form, if not isotrivial, will admit specializations at infinitely many $t\in T$ for which some infinite critical orbit on the generic fibre becomes periodic, for instance due to a result of Dujardin and Favre~\cite{DF}. Fix one such parameter $t\in T$, at which some critical point specializes to a point of period $q$. We then have $W_{q, 0}=0$ at $t$ by inspection of~\eqref{eq:Wintro} and, ignoring the possibility that $W_{n, m}=0$ identically for the moment, one can show by a simple induction that $W_{q(n+1), qn}$ vanishes to order at least $n$ at this parameter. This contradicts the earlier consequence of Haj\'os Lemma, once $n$ is large enough.

There are several problems with the proof sketch above. First, it might happen that some of our Laurent polynomials $W_{n, m}$ vanish identically, for instance if \emph{any} critical point of the generic fibre is preperiodic. One can replace~\eqref{eq:Wintro} with a product over non-preperiodic critical points on the generic fibre of a given family, but restriction to the torsion coset discussed above may introduce more critical orbit collisions. At least as problematic is that the torsion coset might only be defined only over a proper extension of the field $K$, which changes the output of the result of Dvornicich and Zannier, and consequently the order-of-vanishing bound derived from Haj\'os Lemma. Both of these problems are solved by proving the non-existence of families of polynomials with certain properties over non-trivial tori over any number field, by induction on the dimension of the torus.

\subsection*{Formalization} A formalization of a proof of Theorem~\ref{thm:main} from results in the cited works, produced by Codex and Claude Code from a draft of this paper with guidance from the authors, is available at \texttt{https://github.com/pingram3/pcf-cyclotomic-lean} It is worth commenting on what this formalization represents. First, it is a relative formalization, offering a Lean proof of Theorem~\ref{thm:main} \emph{assuming} the truth of the necessary results from the cited works. It is also formalization of the proof of Theorem~\ref{thm:main} in particular, and so does not constitute a Lean proof of any other claim in this paper (although it turns out to go via several of the numbered intermediate claims).

\subsection*{Acknowledgements}
Use of consumer-grade AI tools from Anthropic and OpenAI provided ideas that helped in generalizing earlier versions of this result from one-parameter families to all monic centred polynomials, as well in locating references to prior work and refereeing prior drafts. The authors became aware of the independent work~\cite{grz} during efforts to generalize this result from $K^{\mathrm{cyc}}$ to $K^{\mathrm{ab}}$ and to fully formalize these results.

\section{Laurent polynomials and torsion on tori}

In this section we recall some well-known properties of tori, over an algebraically closed field $K$ of characteristic zero. First, consider the multiplicative group $G=\mathbb{G}_\mathrm{m}^r$ over $K$, whose coordinate ring is the ring of Laurent polynomials $K[u_1^{\pm 1}, ..., u_r^{\pm 1}]$. The torsion in $G$ is precisely the set of points whose coordinates are roots of unity. A theorem of Laurent describes the closure of infinite sets of torsion points.

For a tuple $u=(u_1, ..., u_r)$ of variables, and a tuple $e=(e_1, ..., e_r)$ of integers, we write $u^e$ for $u_1^{e_1}\cdots u_r^{e_r}$, allowing us to write Laurent polynomials in the form $\sum \beta_i u^{e_i}$. Similarly, given a subtorus $H\subseteq G$ and a tuple $x =(x_1, ..., x_n)\in G$, we write $xH$ for the coset of $H$ by $x$. There if $H$ has dimension $r'$, there is an isomorphism (of varieties) $\phi:\mathbb{G}_m^{r'}\to xH$ whose coordinate functions have the form $x_j w^{e_j}$, where $w$ is the tuple of coordinates on $\mathbb{G}_m^{r'}$. Note that if $x$ is torsion, $\phi$ will map the torsion of $\mathbb{G}_m^{r'}$ to the torsion points of $\mathbb{G}_\mathrm{m}^r$ in $xH$.

The following lemma is a consequence of Laurent's Theorem~\cite{Laurent}.
\begin{lemma}\label{thm:laurent}
If $T\subseteq G$ is an infinite set of torsion points, then there is a positive-dimensional torsion coset $\eta H \subseteq G$ such that $\eta H\cap T$ is Zariski dense in $\eta H$.
\end{lemma}

In general, let $V$ and $W$ be affine varieties over $K$, let $x\in V$, and let $\ord_x:K[V]\to \ZZ\cup\{\infty\}$ be defined in the usual way, so that
\[\ord_x(F)=\sup\{k:F\in\mathfrak{m}_x^k\},\]
with $\mathfrak{m}_x\subseteq K[V]$ the maximal ideal of functions vanishing at $x$. It will be useful to note that if $\phi:V\to W$ is a morphism, then for any $F\in K[W]$ we have
\begin{equation}\label{eq:orderpull}\ord_{\phi(x)}(F)\leq \ord_x(\phi^*F)\end{equation}
since $\phi^*(\mathfrak{m}_{\phi(x)})\subseteq \mathfrak{m}_x$.

\section{Adelic bounds for PCF polynomials}

We now recall a result on PCF polynomials. Recall that, given a number field $K$ with the usual set of places $M_K$, and $M_K$-constant is a function $M:M_K\to \mathbb{R}^+$ such that $M_v:=M(v)=1$ for all by finitely many $v\in M_K$. We will extend an $M_K$-constant to an $M_L$ constant, for a finite extension $L/K$, just by setting $M_w=M_v$ when $|\cdot|_w$ on $L$ restricts to $|\cdot|_v$ on $K$. In fact, all of our $M_K$-constants will be extensions of $M_\QQ$-constants.
\begin{proposition}\label{prop:house}
There is an adelic constant $M_v=M(d, v)\geq 1$ with the following property. Let
$f(z)=z^d+a_{d-2}z^{d-2}+\cdots+a_0\in\overline{\QQ}[z]$ be monic centered and
PCF. Then, for every place $v$ and every extension of $|\cdot|_v$ to $\overline{\QQ}$ we have that
\begin{enumerate}
\item\label{it:one} each $|a_i|_v\leq M_v$, and
\item\label{it:two} every critical point $c$ of $f$ and every $n, m\geq 0$, we have $|f^{n}(c)|_v\leq M_v$, and $|f^{n}(c)-f^m(c)|_v\leq M_v$.
\end{enumerate}
\end{proposition}

\begin{proof}
Part~\eqref{it:one} is contained in \cite{Ingram}, but not quite explicitly. Let $\sigma_j = \sigma_j(c_1, \dots, c_{d-1})$ denote the elementary symmetric polynomial of degree $j$, and set
\[g(z)=\frac{1}{d}z^d-\frac{\sigma_1}{d-1}z^{d-1}+\cdots +(-1)^{d-1}\sigma_{d-1}z\]
so that
\[g'(z)=\prod_{i=1}^{d-1}(z-c_i).\]
If $g$ is PCF, then \cite[Lemma~2.5]{Ingram}  shows that
 $|c_i|\leq \Delta_v$, for some $M_\QQ$-constant $\Delta$ (which is $e^{\delta}$ in the notation of that paper), and for every extension of $|\cdot|_v$ to $\overline{\QQ}$ (applying the lemma to a Galois conjugate of $g$ if necessary). Increasing $\Delta$, we may assume that $\Delta_v\geq 1$ for all $v$. Now, let $\psi(z)=\alpha z + \beta$ with $\alpha^{d-1}=d$ and $\beta=(c_1+\cdots +c_{d-1})/(d-1)$, so that \[f(z)=g^{\psi}(z)=\psi^{-1}\circ g\circ \psi(z)=z^d+a_{d-2}z^{d-2}+\cdots +a_0\] is monic and centred. 
Expanding, the coefficients are given by
\[
a_k = \frac{\alpha^{k-1}}{k}
   \sum_{j=0}^{d-k} (-1)^{j} \binom{d-1-j}{k-1} \sigma_j \beta^{d-k-j}
\]
for $1\leq k\leq d-2$, and
\[
a_0 = \frac{g(\beta) - \beta}{\alpha}
=\frac{1}{\alpha} \left( \sum_{m=1}^{d} \frac{(-1)^{d-m}}{m} \sigma_{d-m} \beta^{m} - \beta \right).
\]
So for $1\leq k\leq d-2$ we have for archimedean places
\begin{align}
|a_k|_v&\leq \left|\frac{d^{(k-1)/(d-1)}}{k}\right|_v\sum_{j=0}^{d-k}  \binom{d-1-j}{k-1}\left| \sigma_j\, \beta^{\,d-k-j}\right|_v \nonumber\\
&\leq \left|\frac{d^{(k-1)/(d-1)}}{k}\right|_v\sum_{j=0}^{d-k}  \binom{d-1-j}{k-1}\binom{d-1}{j}\Delta_\infty^{d-k},\label{eq:archak}
\end{align}
and
\begin{equation}
\label{eq:archa0}|a_0|_v \leq d^{-1/(d-1)} \sum_{m=1}^{d}\binom{d-1}{m-1} \frac{1}{m} \Delta_\infty^{d} + \Delta_\infty.
\end{equation}
Choosing $N_v$ to be the maximum of the upper bounds in~\eqref{eq:archak} and~\eqref{eq:archa0} for $1\leq k \leq d-2$, we have $|a_k|_v\leq N_v$ for all $k$ with $N_v$ depending just on $d$.

For $v$ non-archimedean, we have $|\sigma_j|_v\leq \Delta_v^j$ and $|\beta|_v\leq |d-1|_v^{-1}\Delta_v$, and so
\begin{align}
|a_k|_v&\leq \left|\frac{d^{(k-1)/(d-1)}}{k}\right|_v\max_{j=0}^{d-k} \left| \sigma_j\, \beta^{d-k-j}\right|_v\nonumber \\
&\leq \left|\frac{d^{(k-1)/(d-1)}}{k(d-1)^{d-k}}\right|_v\Delta_v^{d-k}\label{eq:naak}
\end{align}
and
\begin{equation}\label{eq:naa0}|a_0|_v \leq |d|_v^{-1/(d-1)}|d-1|_v^{-d}\Delta_v^{d}\max_{1\leq m\leq d}|m|_v^{-1},\end{equation}
and so we may again choose $N_v$ to satisfy $|a_k|_v\leq N_v$ for all $k$, noting that the upper bounds in~\eqref{eq:naak} and~\eqref{eq:naa0} are 1 for those $v$ with $\Delta_v=1$ and $|k|_v=1$ for all $1\leq k\leq d$, which is all but finitely many.

So $|a_i|_v\leq N_v$ for some $M_\QQ$-constant $N$ depending only on $d$ and $\Delta$, and we may take $N_v\geq 1$. But every monic centred polynomial is the image of some $g$ under this map, and so we have our first claim. Indeed, if $f$ is monic and centred, let $p$ be a fixed point of $f$, and let $g=\psi\circ f\circ \psi^{-1}$ with $\psi(z)=\alpha(z-p)$. Then $g(0)=0$ and $g$ has leading coefficient $1/d$, so $g$ has the form above, and since the critical points of $f$ sum to zero, the critical points $c_i$ of $g$ have mean $-\alpha p$, so that $\psi(z)=\alpha z+\beta$ as above. Since being PCF is invariant under conjugation, the bound on the $c_i$ applies.

For~\eqref{it:two} we recall the standard escape argument. If $v$ is non-archimedean and $|a_k|_v\leq N_v$ for all $k$, then $|z|_v>N_v\geq 1$ gives
\[|f(z)|_v=|z|^d_v\cdot |1+a_{d-2}z^{-2}+\cdots +a_0z^{-d}|_v=|z|_v^d>|z|_v,\]
and so by induction $z$ is not preperiodic for $f$. It follows for $f$ PCF that $|f^k(c)|\leq N_v$ for all $k\geq 0$ and all critical points $c$ of $f$.

In the case where $v$ is archimedean, with $|a_k|_v\leq N_v$, we see that $|z|_v> 2(d-1)N_v$ gives 
\[|1+a_{d-2}z^{-2}+\cdots +a_0z^{-d}|_v\geq \frac{1}{2}\]
and hence
\[|f(z)|_v\geq \frac{1}{2} |z|_v^d> |z|_v.\]
Again, such a point cannot be preperiodic, so we must have $|f^k(c)|_v\leq 2(d-1)N_v$ for all $f'(c)=0$ and all $k\geq 0$. So we set $M_v=4(d-1)N_v$ for archimedean places, and $M_v=N_v$ for non-archimedean. By the triangle and ultrametric inequalities, we have both $|f^n(c)|\leq M_v$ and $|f^n(c)-f^m(c)|\leq M_v$ for all $n, m$, and $v$.
\end{proof}

It turns out that adelically bounded sets of elements of $\Kcyc$ can be represented by certain linear forms over $K$, evaluated at points whose coordinates are zero or roots of unity. The following is a modification of a result of Dvornicich and Zannier \cite{DZ}, generalizing work of Loxton \cite{Loxton}.

\begin{lemma}\label{lem:DZ}
Let $K$ be a number field, and let $M_v$ be an adelic constant. Then exists a linear form \begin{equation}\label{eq:loxtonform}L(x)=e_1x_1+\cdots +e_mx_m\in K[x]\end{equation} such that 
every $\alpha\in\Kcyc$ with $|\alpha|_w\leq M_v$ for all $v\in M_K$ and all places $w$ of $K(\alpha)$ above $v$ can be written as $\alpha=L(\zeta_1, ..., \zeta_m)$, with each $\zeta_i$ either zero or a root of unity.
\end{lemma}

\begin{proof}
Dvornicich--Zannier~\cite[Theorem~L]{DZ} prove that for every $B'>0$ there are a finite set $E=\{c_1, ..., c_s\}\subseteq K$ and an integer $N\geq 1$ such that every algebraic integer in $\Kcyc$ all of whose conjugates have absolute value at most $B'$ can be written as $\sum_{\nu=1}^{b}c_{j_\nu}\xi_\nu$ with $b\leq N$, each $c_{j_\nu}\in E$, and each $\xi_\nu$ a root of unity. In general, choose an integer $D\geq 1$ such that $|D|_v\leq M_v^{-1}$ for all finite places $v$, and let $M'=\max_{v\in M_K^\infty} DM_v$. Then our hypothesis implies that $D\alpha$ is an algebraic integer all of whose conjugates have absolute value at most $M'$, so we may write $D\alpha = \sum_{\nu=1}^b c_{j_\nu} \xi_\nu$ with $b\leq N$. Now let $L$ be the linear form in the $sN$ variables $x_{j, k}$, $1\leq j\leq s$ and $1\leq k\leq N$, given by
\[L(x)=\sum_{j=1}^s\sum_{k=1}^N \frac{c_j}{D}x_{j, k}.\]
Placing each term $c_{j_\nu}\xi_\nu$ in an unused variable $x_{j_\nu, k}$, and setting the remaining variables to zero, we obtain $\alpha=L(\zeta)$ with each $\zeta_{j, k}$ either zero or a root of unity.
\end{proof}

\section{Haj\'os' Lemma}

We will need the following higher-dimensional variant of a lemma of Haj\'os. 

\begin{lemma}\label{lem:hajos}
Let $x\in\mathbb{G}_\mathrm{m}^m (\CC)$, and let $F$ be a non-zero Laurent polynomial in $m$ variables expressible as a linear combination of at most $T$ monomials.
Then
\[\ord_x\left(F\right)\leq T-1.\]
\end{lemma}

\begin{proof}
We start with the case $m=1$, which is well-known but we prove for completeness by induction on $T$. Of course, for $T=1$ we have that $F$ is a (non-zero) monomial, and so does not vanish on $\CC^\times$. Now, if the claim is known for $T=k$, consider
$F=\sum_{i=1}^{k+1} \beta_i u^{r_i}$. If any $\beta_i$ vanishes or if the exponents are not distinct, we can re-write this expression with fewer terms and apply the induction hypothesis directly, so without loss of generality suppose that $r_1>r_2\cdots >r_{k+1}$ with $\beta_i\neq 0$ for all $i$. If $F(x)\neq 0$ there is nothing to prove, so suppose $F(x)=0$. For $x\neq 0$ any polynomial $P$ vanishing at $x$ we have $\ord_x(P)=\ord_x(P')+1$, and hence
\begin{multline*}\ord_x\left(F\right)=\ord_x\left(\sum_{i=1}^{k+1} \beta_i u^{r_i-r_{k+1}}\right)\\=\ord_x\left(\sum_{i=1}^{k} (r_i-r_{k+1})\beta_i u^{r_i-r_{k+1}-1}\right)+1\leq k-1+1\end{multline*}
by the induction hypothesis. So the claim is established in dimension 1.

Now, let $F=\sum_{i=1}^T \beta_iu^{r_i}$ in the original context, assume without loss of generality that no $\beta_i$ vanishes, and suppose that $F(x)=0$ for $x=(x_1, ..., x_m)\in (\CC^\times)^m$. For exponents $s_i\in \ZZ$ to be determined, set
\[\psi(t)=(x_1 t^{s_1}, ..., x_mt^{s_m}),\]
and note that
\[F\circ \psi (t)=\sum_{i=1}^T \beta_i x^{r_i} t^{r_i\cdot s},\]
where $r_i\cdot s$ is the usual dot product with $s=(s_1, ..., s_m)$, and
\[\ord_x(F)\leq \ord_{1}(F\circ\psi)\] by~\eqref{eq:orderpull}.
Thus we reduce to the 1-dimensional case as long as $F\circ \psi$ does not vanish identically, but the dot products $r_i\cdot s$ are distinct except on a finite union of hyperplanes in $\ZZ^m$, and so we can choose $s$ so that $F\circ\psi$ has the same number of non-zero terms as $F$.
\end{proof}

\section{Super-attracting parameters}

Let $R$ be a finitely generated $\CC$-algebra domain with field of fractions $F$, and let $\mathfrak{m}$ be a maximal ideal of $R$ with residue field $\CC$. As usual, we identify maximal ideals of $R$ with closed points on $\operatorname{Spec}(R)$. If $f(z)\in F(z)$ is a monic polynomial of degree $d$, we say that $\mathfrak{m}$ is a \emph{super-attracting parameter} if and only if   the coefficients of $f$ are regular at $\mathfrak{m}$ and there is a finite extension $E/F$, a critical point $f'(\zeta)=0$ in $E$ with infinite forward orbit under $f$, and an extension $\mathfrak{n}$ of $\mathfrak{m}$ to $E$ such that $\zeta$ becomes periodic modulo $\mathfrak{n}$.

\begin{lemma}\label{lem:super}
Let $R=\CC[\mathbb{G}_\mathrm{m}^r]$ and let $f(z)\in R[z]$ be monic, centred, and non-isotrivial. Then there is a super-attracting parameter $x\in \mathbb{G}_\mathrm{m}^r(\CC)$ for $f$.
\end{lemma}

\begin{proof}
The claim follows in the stated generality from results of Dujardin and Favre~\cite{DF}, but in the interest of avoiding some technicalities, we apply these results in the case of dimension 1, and conclude the higher-dimensional statement by restricting to a curve.

In the case $r=1$, choose a smooth affine curve $B$ with a finite surjective map $B\to \mathbb{G}_\mathrm{m}$ such that every critical point $c$ of $f$ is defined over $\CC[B]$. In a slight abuse of notation, we will also denote the pull-back of $f$ to $B$ by $f$. We are given that $f$ is non-isotrivial (which is true over $B$ given that it is true over $\mathbb{G}_\mathrm{m}$), so each $(f, c)$ is a holomorphic family of rational maps of degree $d$ with a marked critical point, as per~\cite{DF}. If every critical point $c$ were persistently preperiodic, then every fibre of the family would be PCF, and hence in the connectedness locus, which is compact in the space of monic centred polynomials~\cite{BH}. The coefficients of $f$ would then be bounded regular functions on $B$, hence constant, which is impossible given that $f$ is not isotrivial. So some critical point $c$ has infinite forward orbit on the generic fibre. By~\cite[Theorem~2.5]{DF}, the pair $(f, c)$ is not passive, since it is neither isotrivial nor persistently preperiodic, and then~\cite[Proposition~2.4 and its proof]{DF} gives a parameter $b_0\in B$ at which $c$ is periodic. The image of $b_0$ in $\mathbb{G}_\mathrm{m}$ is a super-attracting parameter.

Now consider the general case, with $a_i\in \CC[u_1^{\pm 1}, ..., u_r^{\pm 1}]$.   If $a_i(u)=\sum_{j=1}^{e_i} \gamma_j u^{r_{i, j}}$, note that the conditions $\beta\cdot r_{i,j}=\beta\cdot r_{i', j'}$ for $r_{i, j}\neq r_{i', j'}$ define finitely many hyperplanes in $\ZZ^r$, so we can choose  $\beta\in\ZZ^r$ with the dot products $\beta\cdot r_{i,j}$ distinct when the $r_{i, j}$ are, and with $\beta\cdot r_{i,j}\neq 0$ when $r_{i,j}\neq 0$. Let $\phi(t)=(t^{\beta_1}, ..., t^{\beta_r})$, which defines a map $\psi:\mathbb{G}_\mathrm{m} \to \mathbb{G}_\mathrm{m}^r$. The polynomials $\phi^*a_i\in \CC[t^{\pm 1}]$ have the same number of distinct monomials as the corresponding $a_i$, and non-constant monomials remain non-constant, and so in particular the $\phi^*a_i$ are not all constant. For monic centred polynomials, being isotrivial is the same as having constant coefficients, so the family $z^d+\phi^*a_{d-2}(t)z^{d-2}+\cdots +\phi^*a_0$ is not isotrivial. Now let $X\to \mathbb{G}_\mathrm{m}^r$ be a finite cover on which the critical points $c_1, ..., c_{d-1}$ of $f$ are regular functions, and let $B$ be the normalization of an irreducible component of $X\times_{\mathbb{G}_\mathrm{m}^r}\mathbb{G}_\mathrm{m}$ (fibre product over $\phi$) dominating $\mathbb{G}_\mathrm{m}$. Applying the previous case on $B$, some $c_i|_B$ has infinite forward orbit and is periodic at some $b_0\in B$. Then $c_i$ has infinite forward orbit on $X$, since a relation $f^n(c_i)=f^m(c_i)$ on $X$ would restrict to $B$, and the image of $b_0$ in $X$ lies above $x=\phi(t_0)$, where $t_0$ is the image of $b_0$ in $\mathbb{G}_\mathrm{m}$. So $x$ is a super-attracting parameter for the original family.
\end{proof}

Note that while Dujardin and Favre~\cite{DF} restrict attention to marked critical points, many of the results were subsequently generalized to arbitrary marked points by DeMarco~\cite{DeMarco}

Haj\'os' Lemma tells us that certain elements of $\CC[u_1^{\pm 1}, ..., u_m^{\pm 1}]$ do not vanish to too high an order. In contrast, super-attracting parameters allow us to construct Laurent polynomials vanishing to very high order, and this will eventually be the source of the main contradiction in the proof of Theorem~\ref{thm:main}. Note that if $c$ is the critical point of $f$ over which becomes periodic at $x$, then $f^k(c)-c$ is a function that vanishes at some point above $x$ on some cover of the parameter space. It will be useful to have a function defined on the original parameter space.

\begin{lemma}\label{lem:W}
Let $R=\CC[u_1^{\pm 1}, ..., u_m^{\pm 1}]$, and let $f(z)=z^d+a_{d-2}z^{d-2}+\cdots + a_0\in R[z]$ be a monic centred polynomial. If $S$ is the multiset of critical points of $f$ with infinite forward orbit (over the algebraic closure of the field of fractions of $R$), and 
\[W_{\ell, k}:=\prod_{c \in S}(f^{\ell}(c)-f^k(c)),\]
then $W_{\ell, k}\in R$, and $W_{\ell, k}\neq 0$ for $\ell\neq k$. If the $a_i$ lie in $k_0[u_1^{\pm 1}, ..., u_m^{\pm 1}]$ for some subfield $k_0\subseteq \CC$, then so does $W_{\ell, k}$. Furthermore, if $x$ is a super-attracting parameter, then there is a  $q$ such that
\[\ord_x(W_{q(\ell+1), q\ell})\geq \ell+1\]
for all $\ell\geq 0$.
\end{lemma}

\begin{proof}
First, we confirm that that $W_{\ell,k}\in R$. Let $F$ be the field of fractions of $R$, and let $\overline{F}$ be its algebraic closure. Note that if $c \in \overline{F}$ is a critical point of $f$ with infinite forward orbit, then $\sigma(c)$ has this property for all $\sigma\in \operatorname{Gal}(\overline{F}/F)$, and so we can factor $f'(z)=A(z)B(z)$ in $F[z]$ such that $B(c)=0$ if and only if the orbit of $c$ under $f$ is finite. As the roots of $f'(z)/d$ are integral over $R$, which is integrally closed, we have in fact $f'(z)=A(z)B(z)$ in $R[z]$, with $A$ monic, without loss of generality. It then follows that
\[W_{\ell, k}:=\prod_{c \in S}(f^{\ell}(c)-f^k(c))=\operatorname{Res}(A(z), f^\ell(z)-f^k(z))\in R,\]
just because the resultant is the determinant of a matrix with entries in $R$. The same argument, with $F$ replaced by the field of fractions of $k_0[u_1^{\pm 1}, ..., u_m^{\pm 1}]$, shows that $W_{\ell, k}$ has coefficients in $k_0$ if the $a_i$ do. We have $W_{\ell, k}\neq 0$ for $\ell\neq k$ simply because none of the terms in the product vanish.

Now, let $x$ be a super-attracting parameter, so that there is a critical point $c$ with infinite forward orbit, hence a root of $A$, and a point above $x$ (on some finite cover) at which the reduction $\bar{c}$ of $c$ is periodic, of period dividing some $q\geq 1$. Write $g=f^q$ and
\[H(X, Y)=\frac{g(X)-g(Y)}{X-Y}\in R[X, Y],\]
so that
\[g^{n+1}(z)-g^n(z)=\left(g^n(z)-g^{n-1}(z)\right)H\left(g^n(z), g^{n-1}(z)\right).\]
Since the resultant is multiplicative, setting
\[Q_n=\operatorname{Res}\left(A(z), H\left(g^n(z), g^{n-1}(z)\right)\right)\in R\]
we have
\begin{equation}\label{eq:prodf}W_{q(n+1), qn}=W_{qn, q(n-1)}Q_n\end{equation}
for all $n\geq 1$. Since $A$ is monic, these resultants are products of values of polynomials at the roots of $A$, and so they vanish at $x$ as soon as the two polynomials have a common root modulo the point above $x$. Now $A(z)$ and $g(z)-z$ have the common root $\bar{c}$ modulo this point, so $W_{q, 0}$ vanishes at $x$. Similarly, $\bar{g}^n(\bar{c})=\bar{g}^{n-1}(\bar{c})=\bar{c}$ and
\[\bar{H}(\bar{c}, \bar{c})=\bar{g}'(\bar{c})=0,\]
since $(f^q)'(z)$ has $f'(z)$ as a factor and $f'(\bar{c})=0$. So $Q_n$ vanishes at $x$ for all $n\geq 1$, and by~\eqref{eq:prodf} and induction
\[\ord_x(W_{q(\ell+1), q\ell})\geq \ell+1\]
for all $\ell\geq 0$.
\end{proof}

\section{Exceptional families}

Fix $d$ and a number field $K$ as above. An \emph{exceptional family} of dimension $r$ will consist of a finite extension $K'/K$, a tuple of Laurent polynomials \[a_{d-2}, ..., a_0\in K'[u_1^{\pm 1}, ..., u_r^{\pm 1}],\] and an infinite set $T\subseteq (\mathbb{G}_\mathrm{m}^r
)^{\mathrm{Tors}}$ such that \[f_{t_0}(z)=z^d+a_{d-2}(t_0)z^{d-2}+\cdots +a_0(t_0)\] is PCF for every $t_0\in T$, and such that the $f_{t_0}$ are distinct for distinct $t_0\in T$. Of course, there can be no exceptional families of dimension $r=0$, since $\mathbb{G}_\mathrm{m}^0$ is the trivial group. The main result of this section is that there aren't any in higher dimension, either, which we will prove by induction on $r$.

First, though, we prove that we can trim down some families that do not quite meet the definition of an exceptional family.

\begin{lemma}\label{lem:trim}
Let $K'/K$ be a finite extension, let $a_{d-2}, ..., a_0\in K'[u_1^{\pm 1}, ..., u_r^{\pm 1}]$, and let $T\subseteq ((K')^\mathrm{cyc})^r$ be an infinite set of points at which the $a_i$ are regular, whose coordinates are all zero or roots of unity, such that $f_{t_0}$ is PCF for every $t_0\in T$, and such that $t_0\mapsto f_{t_0}$ is injective. Then there exists an exceptional family of dimension at most $r$.
\end{lemma}

\begin{proof}
We proceed by induction on the number $n$ of coordinates such that there exists a point of $T$ with a zero in that coordinate. Of course, if $n=0$ then these data already define an exceptional family.

Now suppose that $n\geq 1$, and without loss of generality, suppose that the $r$th coordinate of some point in $T$ contains a zero. Let $T'\subseteq T$ be the subset consisting of those points whose $r$th coordinate is non-zero. If $T'$ is infinite, then we may replace $T$ by $T'$ and find ourselves in the previous case of the induction. Suppose $T'$ is finite. Since each $a_i$ is regular at some point of $T\setminus T'$, where $u_r=0$, no $a_i$ involves a negative power of $u_r$, and so we may set $b_i=a_i(u_1, ..., u_{r-1}, 0)\in K'[u_1^{\pm 1}, ..., u_{r-1}^{\pm 1}]$ and take $S$ to be the projection of $T\setminus T'$ onto the first $r-1$ coordinates. This projection is injective, since the last coordinate vanishes on $T\setminus T'$, so $S$ is infinite, the $b_i$ are regular at the points of $S$, and the specializations of $z^d+b_{d-2}z^{d-2}+\cdots +b_0$ at points of $S$ are exactly the specializations of the original family at points of $T\setminus T'$. So we again find ourselves in the setting of the induction hypothesis.
\end{proof}

\begin{lemma}\label{lem:excep}
There are no exceptional families.
\end{lemma}

Assuming this lemma, we have a direct route to the main result, since an infinite number of monic PCF polynomials allows one to easily build an exceptional family.

\begin{proof}[Proof of Theorem~\ref{thm:main} assuming Lemma~\ref{lem:excep}]
Let $K$ be a number field, and suppose that there are infinitely many monic centred PCF maps with coefficients in $K^{cyc}$. We will construct an exceptional family, in which the $a_i$ are in fact linear forms. 

Let $m$ and $e_1, ..., e_m\in K$ be the quantities ensured by Lemma~\ref{lem:DZ} applied to the adelic constant of Proposition~\ref{prop:house}, and consider the variety
\[Y\subseteq \mathbb{A}^{d-1}\times \mathbb{A}^{(d-1)m}\]
cut out by the equations
\[a_i = \sum_{j=1}^m e_j y_{i, j}.\]
For any tuple $(a_{d-2}, ..., a_0)\in K^{cyc}$ defining a PCF map, there is a point $y\in Y$ whose coordinates $y_{i, j}$ are all either $0$ or roots of unity, and we choose one such of each to obtain a set $T_0$ with the property that each tuple $(a_{d-2}, ..., a_0)$ is the projection of exactly one point in $T$. Applying Lemma~\ref{lem:trim} to the linear forms $\sum_j e_j y_{i, j}$ and the set of $y$-coordinates of points of $T_0$ give us an exceptional family.
\end{proof}

\begin{proof}[Proof of Lemma~\ref{lem:excep}]
Let $M=M(d)$ be the adelic constant from Proposition~\ref{prop:house}.
Suppose that there is at least one exceptional family, and choose one of minimal dimension $r\geq 1$. Thus we have a finite extension $K'/K$, a tuple of Laurent polynomials $a_{d-2}, ..., a_0\in K'[u_1^{\pm 1}, ..., u_r^{\pm 1}]$, and an infinite set $T\subseteq (\mathbb{G}_\mathrm{m}^r)^{\mathrm{Tors}}$ with the properties described above.

By Lemma~\ref{lem:DZ}, applied over $K'$ to the adelic constant $M_v^{d-1}$, there is for some $m'\geq 1$ a linear form
\[L(x_1, ..., x_{m'})=e_1x_1+\cdots + e_{m'}x_{m'}\in K'(x_1, ..., x_{m'})\]
such that every $\alpha\in (K')^\mathrm{cyc}$ with $|\alpha|_v\leq M_v^{d-1}$, for all $v$, can be written as
\begin{equation}\label{eq:reps}\alpha = \sum_{i=1}^{m'} e_i \zeta_i,\end{equation}
with $\zeta_i\in \mu_\infty\cup\{0\}$. Since the $f_{t_0}$ are distinct for $t_0\in T$, and each conjugacy class contains only finitely many monic centred polynomials, the family is not isotrivial, so by Lemma~\ref{lem:super} there is a super-attracting parameter $y$ for $f$, and an associated period $q$. Let $W=W_{q(m'+1), qm'}$ in the notation of Lemma~\ref{lem:W}, and consider the, write $\pi: \mathbb{G}_\mathrm{m}^r\times \mathbb{A}^{m'}\to \mathbb{G}_\mathrm{m}^r$ for projection onto the first coordinate, and consider the equation
\begin{equation}\label{eq:Wpi}W(u)=\sum_{i=1}^{m'} e_i y_i,\end{equation}
By Lemma~\ref{lem:W} with $k_0=K'$ we have $W\in K'[u_1^{\pm 1}, ..., u_r^{\pm 1}]$, so $W(t_0)\in (K')^\mathrm{cyc}$ for any $t_0\in T$. Since $W(t_0)$ is a product of at most $d-1$ differences $f_{t_0}^n(c)-f_{t_0}^m(c)$ with $c$ a critical point of $f_{t_0}$, Proposition~\ref{prop:house} gives $|W(t_0)|_v\leq M_v^{d-1}$ for all $v$, and so there is a point $s\in \pi^{-1}(t_0)$ such that the coordinates $y_i(s)$ consist of 0s and roots of unity. Let $S_0$ be a set constructed by choosing one such $s$ above each $t_0\in T$, noting that $f_{\pi_1(s)}$ are distinct as $s\in S_0$ varies.

Now, let $I_0= \{1, ..., m'\}$. For each $j\geq 1$, if there are infinitely many $s\in S_{j-1}$ with $y_j(s)\neq 0$, let $S_j\subseteq S_{j-1}$ be the set of all such points and let $I_j=I_{j-1}$. Otherwise, let $S_j$ be the set of points $s\in S_{j-1}$ with $y_j(s)=0$ and let $I_j=I_{j-1}\setminus \{j\}$. In the end, we have constructed a, infinite set $S_{m'}\subseteq S_0$ with the property that $y_j(s)\neq 0$ for all $s\in S_{m'}$ and $j\in I_{m'}$, and $y_j=0$ for all $s\in S_{m'}$ and $j\not\in I_{m'}$. For convenience, re-order the indices so that $I_{m'}=\{1, ..., m''\}$ and let $T'=\{(\pi(s), y_1(s), ..., y_{m''}(s)):s\in S_{m'}\}\subseteq \mathbb{G}_\mathrm{m}^r\times \mathbb{G}_\mathrm{m}^{m''}$. We then have
\begin{equation}\label{eq:Ws}W(\pi(s))=\sum_{i=1}^{m''} e_i y_i(s)\end{equation}
for all $s\in S_{m'}$, the coordinates of all points in $T'$ are roots of unity, and the projection $\pi_1:T'\to T$ onto the first $r$ coordinates is injective.

By Laurent's Theorem~\ref{thm:laurent} there is a positive-dimensional coset $\omega\Gamma\subseteq\mathbb{G}_\mathrm{m}^r\times \mathbb{G}_\mathrm{m}^{m''}$ such that $T'\cap \omega\Gamma$ is dense in $\omega\Gamma$.
Suppose first that $P=\pi_1(\omega\Gamma)\neq \mathbb{G}_\mathrm{m}^r$. Then $P$ is a torsion coset of dimension $r'<r$. If $r'=0$ then $P$ is a point, and since $\pi_1$ is injective on $T'$, the set $T'\cap\omega\Gamma$ would contain at most one point, contradicting its density in the positive-dimensional $\omega\Gamma$. So $r'\geq 1$, and there is an isomorphism $\phi:\mathbb{G}_\mathrm{m}^{r'}\to P$, defined over some extension $K''/K'$ obtained by adjoining finitely many roots of unity, specifically a map of the form
\[\phi(w)=(\omega_1 w^{v_1}, ..., \omega_{r}w^{v_{r}}),\]
where $v_i$ are exponent vectors and $\omega_i$ are roots of unity. We claim that the Laurent polynomials $b_i=\phi^* a_i$ and the set $T''=\phi^{-1}(\pi_1(\omega\Gamma\cap T'))$ form an exceptional family of dimension $r'<r$.

To see this, note that $b_i\in K''[y_1^{\pm 1}, ..., y_{r'}^{\pm 1}]$ by construction, and $T''$ is infinite since $\omega\Gamma$ contains infinitely many points of $T'$, and $\pi_1$ is injective on $T'$. The points of $T''$ are torsion, since $\phi$ is an isomorphism given by monomials times roots of unity. Also, the specialization of $z^d+b_{d-2}z^{d-2}+\cdots +b_0$ at $t_0\in T''$ is precisely the specialization of our original family $z^d+a_{d-2}z^{d-2}+\cdots + a_0$ at $\phi(t_0)\in\pi_1(T')\subseteq T$, so specializing at any $t_0\in T''$ yields a PCF polynomial, and the specialization map is injective because $\pi_1\circ \phi$ maps $T''$ injectively into $T$. So we have constructed an exceptional family of dimension $r'$, with $1\leq r'<r$, contradicting the minimality of $r$.

It must have been the case then that $\pi_1(\omega\Gamma)=\mathbb{G}_{\mathrm{m}}^r$. Let $\phi:\mathbb{G}_\mathrm{m}^\sigma\to \omega\Gamma$ be an isomorphism, necessarily with $\sigma\geq r$, and write
\[\phi(w)=(\omega_1 x^{b_1}, ..., \omega_{\sigma}x^{b_{\sigma}}),\]
as above. Since~\eqref{eq:Ws} holds on a Zariski dense subset of $\omega\Gamma$, it holds on all of $\omega\Gamma$. Let $L=W\circ \pi\circ\phi\in K'[x_1^{\pm 1}, ..., x_\sigma^{\pm 1}]$. On the one hand, we have $L=\sum_{i=1}^{m''} e_i \omega_{r+i} x^{b_{r+i}}$. Since $\pi\circ\phi$ is surjective, it is dominant, so pulling back by it is injective and $L\neq 0$ because $W\neq 0$. In particular $m''\geq 1$, and $L$ is a non-zero Laurent polynomial with at most $m''\leq m'$ terms. It follows that $\ord_x(L)\leq m'-1$ for any $x\in \mathbb{G}_{\mathrm{m}}^\sigma$, by Haj\'os' Lemma~\ref{lem:hajos}.

On the other hand, $y\in \mathbb{G}_\mathrm{m}^r$ is the super-attracting parameter chosen above, and $\ord_y(W)\geq m'+1$ by Lemma~\ref{lem:W}. Since $\pi$ and $\phi$ are surjective over the algebraic closure, there is an $x\in \mathbb{G}_\mathrm{m}^\sigma$ with $\pi\circ\phi(x)=y$, and
\[m'-1\geq \ord_x (L)=\ord_x(W\circ\pi\circ\phi)\geq \ord_{\pi\circ\phi(x)}(W)=\ord_y(W)\geq m'+1.\]
This is a contradiction.
\end{proof}

\end{document}